\documentclass[a4paper,12pt]{article}
\usepackage{amsmath}
\usepackage{amssymb}
\usepackage{amsthm}
\usepackage{mathtools}
\usepackage{mathabx}
\usepackage{enumitem}
\usepackage[utf8]{inputenc}
\usepackage[dvipsnames]{xcolor}
\usepackage{hyperref}

\hypersetup{
    colorlinks,
    linkcolor={blue!80!black},
    citecolor={blue!80!black},
    urlcolor={blue!80!black}
}

\newtheorem{theorem}{Theorem}
\newtheorem{definition}[theorem]{Definition}

\newtheorem{lemma}[theorem]{Lemma}

\newcommand{\coleq}{\coloneqq}
\newcommand{\ud}{\mathrm{d}}
\newcommand{\pd}{\partial}
\newcommand{\norm}[1]{\left\lVert#1\right\rVert}
\newcommand{\abso}[1]{\left|#1\right|}

\newcommand{\bN}{\mathbb{N}}
\newcommand{\bR}{\mathbb{R}}
\newcommand{\bE}{\mathbb{E}}
\newcommand{\bF}{\mathbb{F}}
\newcommand{\cD}{\mathcal{D}}

\newcommand{\e}{\varepsilon}
\DeclareMathOperator{\supp}{supp}

\title{On a nonlinear Peetre's theorem\\in full Colombeau algebras}
\author{E.A.~Nigsch\footnote{Faculty of Mathematics, University of Vienna, Oskar-Morgenstern-Platz 1, 1090 Vienna, Austria. e-mail: \href{mailto:eduard.nigsch@univie.ac.at}{eduard.nigsch@univie.ac.at}}}

\begin{document}

\maketitle

\begin{abstract}
We adapt a nonlinear version of Peetre's theorem on local operators in order to investigate representatives of nonlinear generalized functions occurring in the theory of full Colombeau algebras.
\end{abstract}



\section{Preliminaries}

Algebras of of nonlinear generalized functions in the sense of J.F.~Colombeau \cite{Biagioni,ColNew,ColElem,GKOS,MOBook} provide a way to define a meaningful multiplication of arbitrary distributions while at the same time products of smooth functions and the partial derivatives of distribution theory are preserved. This is the best one can obtain in light of L.~Schartz' impossibility result \cite{Schwartz}.

A certain variant of these algebras, namely those which are termed \emph{full} Colombeau algebras, have been gaining more and more importance recently through their role in the development of a coordinate-invariant formulation of nonlinear generalized function algebras suitable for singular differential geometry and nonlinear problems in a geometrical context. We recall that in general, Colombeau algebras are given as quotients of certain basic spaces containing the representatives of generalized functions. In successive steps, these basic spaces have been modified and enlarged in order for the resulting algebras to accomodate certain desired properties \cite{found, global, bigone, papernew}. At one point in this development, the sheaf property could only be obtained in the quotient by imposing so-called \emph{locality conditions} on the elements of the basic space. 

The object of this article is to study representatives of nonlinear generalized functions on an open subset $\Omega \subseteq \bR^n$ which are given by smooth mappings
\[ R \colon C^\infty(\Omega, \cD(\Omega)) \to C^\infty(\Omega) \]
which satisfy the most general of these locality conditions, i.e., which are \emph{local} (Definition \ref{def_local}). Adapting arguments of J.~Slov\'ak from \cite{zbMATH04036782} we obtain a characterization of locality in simpler terms, i.e., $R(\vec\varphi)(x)$ does not depend on the germ of $\vec\varphi$ at $x$ but only on its jet of infinite order at $x$ (Theorem \ref{thm_peetre}). Furthermore, we examine in which sense such mappings $R$ have locally finite order (Theorem \ref{thm_5}). While any distribution is of finite order locally, no comparable statement exists for Colombeau algebras so far; our results are a first step in this direction.

Let use introduce some notation. Throughout this article we will work on open subsets $\Omega_1 \subseteq \bR^n$ and $\Omega_2 \subseteq \bR^{n'}$ with $n,n' \in \bN$ fixed. We employ the usual multiindex notation $\pd^\alpha$, $\alpha!$, $\abso{\alpha}$ etc.\ with differentiation indices $\alpha \in \bN_0^n$. Given an open subset $\Omega \subseteq \bR^n$ and a locally convex space $\bE$, the space $C^\infty(\Omega, \bE)$ is endowed with its standard topology, which is that of uniform convergence on compact sets in all derivatives separately. For a function $f$ we denote by $j^r f(x)$ the $r$-jet of $f$ at $x$, i.e., the family $(\pd^\alpha f(x))_{\abso{\alpha} \le r}$, where also $r = \infty$ is allowed. The interior of a set $B$ is denoted by $B^\circ$. Note that for smooth functions $f(x,y)$ of two variables we will also write $f(x)(y)$, justified by the exponential law \cite[3.12, p.~30]{KM}.

The formulation of Theorem \ref{thm_5} requires a notion of smoothness for mappings between arbitrary locally convex spaces. The setting we use for this is that of convenient calculus \cite{KM}, i.e., a mapping $f \colon \bE \to \bF$ between two locally convex spaces is said to be smooth in this sense if it maps each smoothly parametrized curve into $\bE$ to a smoothly parametrized curve into $\bF$, i.e., for all $c \in C^\infty(\bR, \bE)$ we have $f \circ c \in C^\infty(\bR, \bF)$.

For convenience we cite the extension theorem of Whitney \cite[1.5.6, p.~31]{zbMATH03934056} which will be heavily used below.
\begin{theorem}[Whitney]\label{whitney}
 Let $\Omega$ be an open subset of $\bR^n$ and $X$ a closed subset of $\Omega$. Given a continuous function $f^\alpha$ on $X$ for each $\alpha \in \bN_0^n$, there exists a function $f \in C^\infty(\Omega)$ with $\pd^\alpha f|_X = f^\alpha$ for all $\alpha \in \bN_0^n$ if and only if for all integers $m \ge 0$ and all compact subsets $K \subseteq X$ we have
\begin{equation}\label{whitcond}
f^\alpha(y) = \sum_{\abso{\beta} \le m} \frac{1}{\beta!} f^{\alpha + \beta} (x) (y-x)^\beta + o ( \norm{y-x}^m)
\end{equation}
uniformly for $x,y \in K$ as $\norm{y-x} \to 0$.
\end{theorem}

\section{Main Results}

We first recall the definition of locality for elements of the basic space
\[ C^\infty(C^\infty(\Omega, \cD(\Omega)), C^\infty(\Omega)) \]
given in \cite{papernew}. While only the case $\Omega_1 = \Omega_2$ was considered there, we use a slightly more general formulation which will be needed below.

\begin{definition}\label{def_local}A mapping $R \colon C^\infty(\Omega_1, \cD(\Omega_2)) \to C^\infty(\Omega_1)$ is called \emph{local} if for all $x \in \Omega_1$ and all $\vec\varphi, \vec\psi \in C^\infty(\Omega_1, \cD(\Omega_2))$ the following implication holds:
\[ \vec\varphi|_U = \vec\psi|_U\textrm{ for some open neighborhood $U$ of }x \Longrightarrow R(\vec\varphi)(x) = R(\vec\psi)(x). \]
\end{definition}

Our first result is the following.

\begin{theorem}\label{thm_peetre} A mapping $R \colon C^\infty( \Omega_1, \cD(\Omega_2)) \to C^\infty(\Omega_1)$ is local if and only if for every $\vec\varphi \in C^\infty(\Omega_1, \cD(\Omega_2))$ and every point $x \in \Omega_1$, $R(\vec\varphi)(x)$ depends on the $\infty$-jet $j^\infty(\vec\varphi)(x)$ only, i.e., if for all $x \in \Omega_1$ and $\vec\varphi, \vec\psi \in C^\infty(\Omega_1, \cD(\Omega_2))$ the equality $j^\infty \vec\varphi (x) = j^\infty \vec\psi (x)$ implies $R(\vec\varphi)(x) = R(\vec\psi)(x)$.
\end{theorem}

The proof imitates that of \cite[Theorem 1, p.~274]{zbMATH04036782} but is adapted in order to incorporate the additional variable $y$ of the smoothing kernels $\vec\varphi(x)(y)$.

\begin{proof}
Suppose we are given $\vec \varphi, \vec\psi \in C^\infty(\Omega_1, \cD(\Omega_2))$ such that $(\pd_x^\alpha \vec\varphi)(x) = (\pd_x^\alpha \vec\psi)(x)$ for some fixed $x \in \Omega_1$ and all $\alpha \in \bN_0^n$. Choose an open neighborhood $W$ of $x$ which is convex and relatively compact in $\Omega_1$, as well as compact sets $K,L \subseteq \Omega_2$ with $K \subseteq L^\circ$ such that $\supp \vec\varphi(a) \cup \supp \vec\psi(a) \subseteq K$ for all $a \in W$.

Next, we construct a sequence $(x_k)_{k \in \bN}$ in $W$ and an open neighborhood $U_k$ of each $x_k$ with $\overline{U}_k \subseteq W$ such that for all $k$ the following conditions hold:
\begin{gather}
\norm{a-x} < \norm{b-x}/2 \quad \forall a \in \overline{U}_{k+1}, b \in \overline{U}_k \label{cond1} \\
\abso{(\pd_x^\alpha \pd_y^\beta \vec\varphi)(a,\xi) - (\pd_x^\alpha \pd_y^\beta \vec\psi)(a,\xi)} \le \frac{1}{k} \norm{a-x}^m \label{cond2} \\
\nonumber\forall a \in \overline{U}_k, \xi \in L, \abso{\alpha} + \abso{\beta} + m \le k.
\end{gather}
It suffices to show that \eqref{cond2} holds for any fixed $\alpha, \beta \in \bN_0^n$ and $m \in \bN$ uniformly for all $\xi \in L$ if $\norm{a-x}$ is small enough. By Taylor's theorem we have for any $f=f(x,y) \in C^\infty(\Omega_1 \times \Omega_2)$, $m \in \bN_0$, $a \in W$ and $\xi \in L$:
\begin{align*}
 f(a,\xi) &= \sum_{\abso{\gamma} < m} \frac{ (\pd_x^\gamma f)(x,\xi)}{\gamma!} (a-x)^\gamma \\
&\quad + m \sum_{\abso{\gamma} = m} \frac{ (a-x)^\gamma }{\gamma!} \int_0^1 (1-t)^{m-1} (\pd_x^\gamma f) ( x + t(a-x), \xi)  \,\ud t.
\end{align*}
Replacing $f$ by $\pd^\alpha_x \pd^\beta_y \vec\varphi - \pd^\alpha_x \pd^\beta_y \vec\psi$ we see that
\begin{gather*}
  (\pd^\alpha_x \pd^\beta_y \vec\varphi - \pd^\alpha_x \pd^\beta_y \vec\psi)(a,\xi) = \\
 m \sum_{\abso{\gamma} = m} \frac{ (a-x)^\gamma }{\gamma!} \int_0^1 (1-t)^{m-1} ( \pd^{\alpha+\gamma}_x \pd^\beta_y \vec\varphi - \pd^{\alpha+\gamma}_x \pd^\beta_y \vec\psi ) ( x + t(a-x), \xi)  \,\ud t \\
= o(\norm{a-x}^m) \qquad \textrm{as }\norm{a-x} \to 0
\end{gather*}
uniformly for $(a,\xi) \in W \times L$. In fact, $(\pd^{\alpha+\gamma}_x \pd^\beta_y \vec\varphi - \pd^{\alpha+\gamma}_x \pd^\beta_y \vec\psi)(a, \xi)$ vanishes for $a=x$ by assumption and is uniformly continuous on the compact set $\overline{W} \times L$, hence the integrand converges to zero uniformly for $\xi \in L$ as $a$ and hence $x + t(a-x)$ approaches $x$.

Note that \eqref{cond1} implies
\begin{equation}\label{cond3}
\norm{a-x} < 2 \norm{a-b}\quad \forall a \in \overline{U}_k, b \in \overline{U}_j, k \ne j.
\end{equation}
In fact, for $k>j$ we have $2 \norm{a-x} < \norm{b-x} \le \norm{a-b} + \norm{a-x}$ and for $k<j$ we have $\norm{a-x} \le \norm{a-b} + \norm{b-x} < \norm{a-b} + \norm{a-x}/2$. Moreover, $x_k \to x$ for $k \to \infty$.

With $A \coleq \{x\} \cup \bigcup_{k} \overline{U}_k$, which is a compact subset of $W$, we define a family of continuous functions $h^{\alpha,\beta}$ on $A \times L$ with $\alpha,\beta \in \bN_0^n$ by
\begin{equation}\label{blubber}
h^{\alpha,\beta}(a,\xi) \coleq \left\{
\begin{aligned}
 (\pd_x^\alpha \pd_y^\beta \vec\varphi)(a,\xi)\qquad &a = x \textrm{ or }a \in \overline{U}_{2k} \textrm{ for some }k, \\
 (\pd_x^\alpha \pd_y^\beta \vec\psi)(a,\xi)\qquad &a \in \overline{U}_{2k+1} \textrm{ for some }k.
\end{aligned}
\right.
\end{equation}
In order to apply Whitney's theorem to this family we have to verify that
\begin{equation}\label{verify}
h^{\alpha,\beta}(b,\eta) = \sum_{\abso{(\gamma,\lambda)} \le m } \frac{ h^{\alpha + \gamma, \beta + \lambda} (a,\xi) } { (\gamma,\lambda)! } (b-a)^\gamma (\eta-\xi)^\lambda + o ( \norm{ (b-a,\eta-\xi) }^m )
\end{equation}
uniformly for $(b,\eta)$ and $(a,\xi)$ in $A \times L$ as $\norm{ (b-a,\eta-\xi) } \to 0$. This follows easily from Taylor's theorem, \eqref{cond2} and \eqref{cond3}.

Consequently, there is a function $\tilde h \in C^\infty(\bR^n \times \bR^{n'})$ whose derivatives on $A \times L$ are given by $\pd_x^\alpha\pd_y^\beta \tilde h = h^{\alpha,\beta}$. Choosing $\rho \in \cD(\Omega_2)$ such that $\rho \equiv 1$ in an open neighborhood of $K$ and $\supp \rho \subseteq L$, set $h(a)(\xi) \coleq \tilde h(a,\xi) \cdot \rho(\xi)$ for $a \in \Omega_1$ and $\xi \in \Omega_2$. Then $h \in C^\infty(\Omega_1, \cD(\Omega_2))$ and 
\begin{gather*}
h|_{U_{2k}} = \vec\varphi|_{U_{2k}},\quad h|_{U_{2k+1}} = \vec\psi|_{U_{2k+1}}\qquad \forall k \in \bN, \\
(\pd_x^\alpha h)(x) = (\pd_x^\alpha\vec\varphi)(x) = (\pd_x^\alpha\vec\psi)(x)\qquad \forall \alpha \in \bN_0^n.
\end{gather*}
The claim of the theorem then follows by
\begin{align*}
R(\vec\varphi)(x) &= \lim_{k \to \infty} R(\vec\varphi)(x_{2k}) = \lim_{k \to \infty} R(h)(x_{2k}) \\
&= \lim_{k \to \infty} R(h)(x_{2k+1}) = \lim_{k \to \infty} R(\vec\psi)(x_{2k+1}) = R(\vec\psi)(x).\qedhere
\end{align*}
\end{proof}

In order to show that $R(\vec\varphi)(x)$ locally depends only on \emph{finitely many} derivatives of $\vec\varphi(x)$ in a certain sense, we will employ the following lemma, paralleling \cite[Lemma 1, p.~276]{zbMATH04036782}.

\begin{lemma}\label{thelemma}
Let $R \colon C^\infty ( \Omega_1, \cD(\Omega_2) ) \to C^\infty(\Omega_1)$ be local and suppose we are given $f \in C^\infty( \Omega_1, \cD(\Omega_2))$, $x_0 \in \Omega_1$ and $K \subseteq \Omega_2$ compact with $\supp f (x_0) \subseteq K$. 

Define $\e \colon \bR^n \to \bR$ by
\[
 \e(x) \coleq \left\{ \begin{aligned}
& \exp ( -1 / \norm{x - x_0}) \qquad & x &\ne x_0, \\
&0 & \qquad x &= x_0.
 \end{aligned}
\right.
 \]

Then there exist a neighborhood $V$ of $x_0$ in $\Omega_1$ and $r \in \bN$ such that for any $x \in V$ and $g_1, g_2 \in C^\infty(\Omega_1, \cD(\Omega_2))$ satisfying
\begin{enumerate}[label=(\roman*)]
 \item $\supp g_i (y) \subseteq K$ for $y$ in a neighborhood of $x$ and $i=1,2$,
 \item $\sup_{\xi \in \Omega_2} \abso{ \pd_x^\alpha \pd_y^\beta ( g_i-f)(x)(\xi) } \le \e(x)$ for $i=1,2$ and $0 \le \abso{\alpha} + \abso{\beta} \le r$
\end{enumerate}
we have the implication $j^r g_1 (x) = j^r g_2(x) \Longrightarrow (Rg_1)(x) = (Rg_2)(x)$.
\end{lemma}

\begin{proof}
Let $R$, $f$, $x_0$ and $K$ be as stated and suppose that the claim does not hold. Then we can find a sequence $x_k \to x_0$ and, for each $k \in \bN$, functions $f_k, g_k \in C^\infty(\Omega_1, \cD(\Omega_2))$ satisfying
\begin{align}
 & \supp f_k(y) \cup \supp g_k(y) \subseteq K \textrm{ for $y$ in a neighborhood of $x_k$}, \\
 & \sup_{\xi \in \Omega_2} \abso{ \pd_x^\alpha \pd_y^\beta ( f_k-f)(x_k)(\xi) } \le \e(x_k),\textrm{ and} \label{dienstag1} \\
 & \sup_{\xi \in \Omega_2} \abso{ \pd_x^\alpha \pd_y^\beta ( g_k-f)(x_k)(\xi) } \le \e(x_k) \label{dienstag2} 
\end{align}
for $0 \le \abso{\alpha} + \abso{\beta} \le k$ such that
\begin{equation}\label{cond7}
j^k f_k(x_k) = j^k g_k(x_k), \quad (Rf_k)(x_k) \ne (Rg_k)(x_k).
\end{equation}
Taking suitable subsequences we may assume that
\[ \norm{x_{k+1} - x_0} \le \norm{x_k - x_0}/2 \]
for all $k \in \bN$ and that all $x_k$ are contained in an open neighborhood $W$ of $x_0$ which is relatively compact in $\Omega_1$ and convex. Furthermore, we can assume that either $x_k \ne x_0$ or $x_k = x_0$ holds for all $k \in \bN$.

In the first case, choose points $y_k \in W$ with $x_0 \ne y_k \ne x_j$ for all $k,j \in \bN$ such that
\begin{align}
  \norm{y_k - x_k} & \le \frac{1}{k} \norm{x_k-x_0}, \label{cond4} \\
\sup_{\xi \in \Omega} \abso{ \pd_x^\alpha \pd_y^\beta (f_k-f)(y_k)(\xi) } & \le 2\e(x_k)\quad (0 \le \abso{\alpha} + \abso{\beta} \le k), \label{cond5} \\
\abso{(Rg_k)(x_k) - (R f_k)(y_k) } & \ge k \norm{x_k - y_k}^{1/k}, \label{cond6} \\
\supp f_k(y_k) & \subseteq K. \label{cond6b}
\end{align}
Such points $y_k$ can be chosen if each of these finitely many conditions holds for $y_k$ in some neighborhood of $x_k$. Conditions \eqref{cond4} and \eqref{cond6b} obviously are without problems. For condition \eqref{cond5} with fixed $\alpha$ and $\beta$ we note that $\pd_x^\alpha\pd_y^\beta (f_k-f)(W)$ is relatively compact (i.e., bounded) in $\cD(\Omega_2)$. Hence, there exists a compact set $B_k \subseteq \Omega_2$ such that $\supp \pd_x^\alpha \pd_y^\beta (f_k - f)(W) \subseteq B_k$. In particular, $\pd_x^\alpha \pd_y^\beta (f_k-f)$ is uniformly continuous in $\overline{W} \times B_k$ and requirement \eqref{cond5} is satisfied for $y_k$ in a small enough neighborhood of $x_k$. Finally, for \eqref{cond6} we first note that by \eqref{cond7} there is $\delta>0$ such that $\abso{(R f_k)(y_k) - (R g_k)(x_k ) } \ge \delta$ for $y_k$ in a small neighborhood of $x_k$ by continuity of $R f_k$. Moreover, we have
\[ k \norm{x_k - y_k}^{1/k} \le \delta \Longleftrightarrow \norm{x_k - y_k} \le (\delta/k)^k \]
which gives \eqref{cond6} for $y_k$ in a small enough neighborhood of $x_k$.

Next, we want to construct a function $h \in C^\infty(\Omega_1, \cD(\Omega_2))$ such that
\begin{equation}
(\pd_x^\alpha\pd_y^\beta h)(x)(\xi) = \left\{
\begin{aligned}
& (\pd_x^\alpha\pd_y^\beta g_k)(x)(\xi) & \qquad &x=x_k \textrm{ for some }k \in \bN,\\
& (\pd_x^\alpha\pd_y^\beta f_k)(x)(\xi) & \qquad &x=y_k \textrm{ for some }k \in \bN,\\
& (\pd_x^\alpha\pd_y^\beta f)(x)(\xi) & \qquad  &x = x_0
\end{aligned}
\right.\label{rhs}
\end{equation}
for all $\alpha,\beta \in \bN_0^n$ and $\xi \in \Omega_2$. For this purpose we apply Whitney's extension theorem to the family $h^{\alpha,\beta}(x)(\xi)$ defined by the right hand side of \eqref{rhs} for $(x, \xi)$ in the compact set $A \times L$ where $A \coleq \{x_0\} \cup \{ x_k : k \in \bN \} \cup \{ y_k : k \in \bN \}$ and the compact set $L \subseteq \Omega_2$ is chosen such that $K \subseteq L^\circ$. Again, we have to verify \eqref{verify}, which is straightforward using Taylor's formula in combination with  \eqref{dienstag1}, \eqref{dienstag2} and \eqref{cond5}.

Emplying a cut-off function as in the proof of Theorem \ref{thm_peetre}, we obtain $h \in C^\infty(\Omega_1, \cD(\Omega_2))$ as desired. Theorem \ref{thm_peetre} now implies
\[ \abso{ (Rh)(x_k) - (Rh)(y_k) } = \abso{ (R g_k) (x_k) - (R f_k) (y_k) } \ge k \norm{x_k-y_k}^{1/k}. \]
For large $k$ this gives a contradiction because $Rh$ is smooth and a fortiori locally Hölder continuous.

In the other case, i.e., $x_k = x_0$ for all $k$, our assumptions imply that
\begin{gather}
\pd_x^\alpha \pd_y^\beta f_k(x_0)(\xi) = \pd_x^\alpha \pd_y^\beta g_k(x_0)(\xi) = \pd_x^\alpha\pd_y^\beta f(x_0)(\xi) \label{blahber1} \\
(R f_k)(x_0) \ne (R g_k)(x_0) \label{blahber}
\end{gather}
for all $k \in \bN$, $\xi \in \Omega_2$ and $\abso{\alpha} + \abso{\beta} \le k$. 

Either $(Rf_k)(x_0)$ or $(Rg_k)(x_0)$ must be different from $(Rf)(x_0)$ for infinitely many values of $k$, hence without loss of generality we can assume that $(R f_k)(x_0) \ne (R f)(x_0)$. As in the previous case, we then choose a sequence $y_k \to x_0$ in an open convex neighborhood $W$ of $x_0$ which is relatively compact in $\Omega_1$ such that
\begin{align}
 \abso{ (Rf_k)(y_k) - (R f)(x_0) } & \ge k \norm { y_k - x_0 }^{1/k}, \\
 \norm{y_{k+1} - x_0} & < \norm{y_k - x_0}/2,\textrm{ and} \\
 \sup_{\xi \in \Omega_2} \abso{ \pd_x^\alpha \pd_y^\beta ( f_k - f)(y_k)(\xi)} & \le \frac{1}{k} \norm{y_k - x_0}^m \label{blahblu}
\end{align}
for all $\abso{\alpha} + \abso{\beta} + m \le k$. \eqref{blahblu} is obtained using Taylor's theorem as in the proof of Theorem \ref{thm_peetre}. Again using Whitney's extension theorem together with a cut-off function in $\cD(\Omega_2)$, we can construct a mapping $h \in C^\infty( \Omega_1, \cD(\Omega_2))$ satisfying
\[ j^\infty h (y_k) = j^\infty f_k(y_k), \quad j^\infty h(x_0) = j^\infty f(x_0). \]
To summarize, by Theorem \ref{thm_peetre} we obtain
\[ \abso{ (R h)(y_k) - (R h)(x_0) } = \abso{ (R f_k)(y_k) - (R f) (x_0) } \ge k \norm{y_k - x_0}^{1/k} \]
in contradiction to Hölder continuity of $Rh$, which concludes the proof.
\end{proof}

With this in place we are able to show the following (cf.~\cite[Theorem 3, p.~278]{zbMATH04036782}):

\begin{theorem}\label{thm_5}Let $R \colon C^\infty(\Omega_1, \cD(\Omega_2)) \to C^\infty(\Omega_1)$ be local and smooth and suppose we are given $f \in C^\infty(\Omega_1, \cD(\Omega_2))$, $x_0 \in \Omega_1$ and a compact subset $K \subseteq \Omega_2$ such that $\supp f(x_0) \subseteq K$. Then there are $r \in \bN$, a neighborhood $V$ of $x_0$ and $\kappa>0$ such that for all $x \in V$ and $g_1, g_2 \in C^\infty(\Omega_1, \cD(\Omega_2))$ with
\begin{enumerate}[label=(\roman*)]
 \item $\supp g_i(y) \subseteq K$ for $y$ in a neighborhood of $x$ and $i=1,2$,
 \item $\sup_{\xi \in \Omega_2} \abso{ \pd_x^\alpha \pd_y^\beta ( g_i - f)(x)(\xi)} \le \kappa$ for $i=1,2$ and $0 \le \abso{\alpha} + \abso{\beta} \le r$,
\end{enumerate}
the condition $j^r g_1(x) = j^r g_2(x)$ implies $(Rg_1)(x) = (Rg_2)(x)$.
\end{theorem}
\begin{proof}
 Fix $R$, $f$, $x_0$ and $K$ as stated and assume the claim does not hold. With $r(k) \coleq 2^{-k}$ there exists a sequence $x_k \to x_0$ and $f_k, g_k \in C^\infty(\Omega_1, \cD(\Omega_2))$ with
\begin{align*}
& \supp f_k(y) \cup \supp g_k(y) \subseteq K \textrm{ for $y$ in a neighborhood of $x_k$}, \\
& \sup_{\xi \in \Omega_2} \abso{ \pd_x^\alpha \pd_y^\beta ( f_k - f)(x_k)(\xi)} \le e^{-r(k)},\textrm{ and} \\
& \sup_{\xi \in \Omega_2} \abso{\pd_x^\alpha \pd_y^\beta ( g_k - f)(x_k)(\xi)} \le e^{-r(k)} 
\end{align*}
for $0 \le \abso{\alpha} + \abso{\beta} \le k$ such that
\begin{equation}\label{cond8}
j^k g_k(x_k) = j^k f_k(x_k),\quad (R g_k)(x_k) \ne (R f_k)(x_k).
\end{equation}
We may assume that $\norm{x_{k+1} - x_0} \le \norm{ x_k - x_0 }/2$. We then construct $s \in C^\infty(\bR \times \Omega_1, \cD(\Omega_2))$ such that
\begin{align*}
 (\pd_x^\alpha \pd_y^\beta s)(2^{-k}, x_k, \xi ) &= (\pd_x^\alpha \pd_y^\beta f_k)(x_k)(\xi), \\
 (\pd_x^\alpha \pd_y^\beta s)(0, x_0, \xi) &= (\pd_x^\alpha \pd_y^\beta f)(x_0)(\xi)
\end{align*}
for all $k \in \bN$, $\xi \in \Omega_2$ and $\alpha,\beta \in \bN_0^n$. Note that $A \coleq \{ ( 0, x_0 ) \} \cup \{ ( 2^{-k}, x_k) : k \in \bN \}$ is compact in $\bR \times \bR^n$. The function $s$ is obtained by applying Whitney's theorem to the family of functions $s^{l,\alpha,\beta}$ (with $l \in \bN_0$ and $\alpha,\beta \in \bN_0^n$) defined on $A \times L$, where $L \subseteq \Omega_2$ is any compact set such that $K \subseteq L^\circ$, by
\[
s^{l, \alpha, \beta}(t,x,\xi) \coleq \left\{ 
\begin{aligned}
& (\pd_x^\alpha \pd_y^\beta f_k)(x_k)(\xi) & &l=0 \textrm{ and } (t,x) = (2^{-k}, x_k)\textrm{ for some }k, \\
& (\pd_x^\alpha \pd_y^\beta f)(x_0)(\xi) & &l=0 \textrm{ and } (t,x) = (0, x_0), \\
&0 \quad& &l \ne 0.
\end{aligned}
\right.
\]
The requirements for Whitney's theorem then are easily verified and we obtain $\tilde s \in C^\infty(\bR \times \bR^n \times \bR^{n'})$ which, by multiplying it with a suitable cut-off function in $\cD(\Omega_2)$ as before, gives $s$ as desired. Next, we define a map
\begin{align*}
 \widetilde R \colon C^\infty( \bR \times \Omega_1, \cD(\Omega_2)) &\to C^\infty(\bR \times \Omega_1) \\
(\widetilde R h)(t,x) &\coleq R(h_t)(x)
\end{align*}
where $h_t \in C^\infty(\Omega_1, \cD(\Omega_2))$ is given by $h_t(x) \coleq h(t,x)$. Obviously, $\widetilde R$ is local in the sense of Definition \ref{def_local}. Now $(R f_k)(x_k) = R(s(2^{-k}, .))(x_k)$ holds because $\pd_x^\alpha f_k(x_k, \xi) = \pd_x^\alpha s(2^{-k}, x_k, \xi)$ for all $\alpha,\beta \in \bN_0^n$ and $\xi \in \Omega_2$ by the construction of $s$ above. Furthermore, $R(s(2^{-k},.))(x_k) = (\widetilde R s)(2^{-k}, x_k)$ by the definition of $\widetilde R$. We now define $\tilde g_k \in C^\infty(\bR \times \Omega_1, \cD(\Omega_2))$ by $\tilde g_k(t,x,\xi) = g_k ( x, \xi)$ and see that
\[ (\pd_t^l \pd_x^\alpha s)(2^{-k}, x_k, \xi) = (\pd_t^l \pd_x^\alpha \tilde g_k)(2^{-k}, x_k, \xi) \]
for $0 \le l + \abso{\alpha} \le k$. Hence, $(\widetilde R s)(2^{-k}, x_k) = (\widetilde R \tilde g_k)(2^{-k}, x_k)$ holds for large values of $k$ by Lemma \ref{thelemma}. Finally, $(\widetilde R \tilde g_k)(2^{-k}, x_k) = (Rg_k)(x_k)$. To summarize, we obtain $(Rf_k)(x_k) = (Rg_k)(x_k)$ for large $k$, which contradicts \eqref{cond8} and concludes the proof.
\end{proof}

\section{Conclusion}

We have seen in Theorem \ref{thm_peetre} that $R(\vec\varphi)(x_0)$ does not depend on the entire germ of $\vec\varphi$ at $x_0$, but only on its $\infty$-jet. Moreover, the statement of Theorem \ref{thm_5} may be reworded as follows: if $R$ is smooth and local and we are given $\vec\varphi$ and $x_0$, there is a neighborhood of $(\vec\varphi,x_0)$ and a natural number $r$ such that that for all $(\vec\psi, x)$ in this neighborhood, the value of $R(\vec\psi)(x)$ depends only on the $r$-jet of $\vec\psi$ at $x$.

\subsection*{Acknowledgments} This work was supported by the Austrian Science Fund (FWF) project P23714.

\end{document}